\documentclass[11pt]{amsart}

\usepackage{amssymb}
\usepackage{overpic}
\usepackage{enumitem}   
\usepackage{graphicx} 
\usepackage{multicol}
\usepackage{xcolor}
\usepackage{amsrefs}
\usepackage[a4paper,margin=1.5cm]{geometry}
\usepackage[hidelinks]{hyperref}

\graphicspath{{Figures/}}

\newtheorem{theorem}{Theorem}
\newtheorem{question}{Question}
\newtheorem{proposition}{Proposition} 
\newtheorem{corollary}{Corollary} 
\newtheorem{remark}{Remark} 
 
\newtheorem{definition}{Definition}

\begin{document}
	
\title[Non-hyperbolic limit cycles on planar polynomial vector fields]{On the structural instability of non-hyperbolic limit cycles on planar polynomial vector fields}

\author[Paulo Santana]
{Paulo Santana$^1$}

\dedicatory{To Jorge Sotomayor, in memorian.}

\address{$^1$ IBILCE-UNESP, CEP 15054--000, S\~ao Jos\'e do Rio Preto, S\~ao Paulo, Brazil}
\email{paulo.santana@unesp.br}



\begin{abstract}
	It is known that non-hyperbolic limit cycles are structurally unstable in the set of planar smooth and analytical vector fields. In the polynomial case, it is known only that limit cycles of even degree are structurally unstable. In this paper, we prove that non-hyperbolic limit cycles of odd degree are also structurally unstable in the polynomial case, if we consider Whitney's topology.
\end{abstract}

\maketitle

\section{Introduction}

Roughly speaking, a vector field $X$ is structurally stable if small perturbations does not change the topological character of its orbits. Let $B^2=\{x\in\mathbb{R}^2\colon ||x||\leqslant1\}$ denote the closed unit disk and let $\partial B^2$ denote its topological boundary. The notion of \emph{structural stability} (first known as \emph{robustness}) is due to Andronov and Pontrjagin \cite{AndPon1937}, whose in 1937 enunciated (without proofs) sufficient and necessary conditions for an analytical vector field $X$, traversal to $\partial B^2$, to be structurally stable in $B^2$. In 1952 DeBaggis \cite{Bag1952} provided the omitted proofs under the less restrictive hypothesis of $X$ being of class $C^1$. On January of 1959 M. M. Peixoto \cite{Pei1959} provided an equivalent definition of structural stability and extend some of the previous results to $B^n=\{x\in\mathbb{R}^n\colon ||x||\leqslant1\}$. On June of 1959 M. M. Peixoto and his wife, M. C. Peixoto \cite{PeiPei1959}, extended the notion of structural stability to vector fields of class $C^1$ defined on a two-dimensional manifold $M\subset\mathbb{R}^2$ with boundary and corners, allowing contact between $X$ and $\partial M$. In 1962 M. M. Peixoto \cite{Pei1962} (from now on referred only as Peixoto) provided sufficient and necessary conditions for structural stability on vector fields of class $C^1$ defined on a two-dimensional closed manifold $M$ (i.e. compact and without boundary). Such characterization is known as \emph{Peixoto's Theorem}. In 1973 Peixoto \cite{Pei1973} also derived a relation between structural stability and Graph Theory. Given a two-dimensional closed manifold $M$, let $\mathfrak{X}$ be the family of all $C^1$-vector fields defined over $M$. Let also $\Sigma\subset\mathfrak{X}$ be the set of all the structural stable vector fields over $M$ and denote $\mathfrak{X}_1=\mathfrak{X}\backslash\Sigma$. In 1974 Sotomayor \cite{Soto1974} provided the complete characterization of all the structurally stable vector fields of $\mathfrak{X}_1$, giving rise to the notion of \emph{structural stability of first order} (or vector fields of \emph{codimension one}). In 1977 Teixeira \cite{Tei1977} extended the work of Sotomayor, allowing $M$ to be a manifold with boundary. From there on, the notion of structural stability flourished in many ways, with the characterization of structural stability of $C^1$-vector fields on open surfaces \cite{Kot1982}, of polynomial vector fields on compact \cites{Soto1985,Cai1979,San1977,Vel} and non-compact \cites{Sha1987,DumSha1990} two dimensional manifolds, $C^1$-vector fields on compact $n$-dimensional manifolds \cite{Mar1961}, gradient flows \cite{Sha1990}, polynomial foliations \cite{JarLliSha2005}, Game Theory \cite{BasBuzSan2024}, quasihomogenous vector fields \cites{OliZha2014,AlgFueGamGar2018}, reversible vector fields \cite{BuzRobTei2021} and piecewise smooth vector fields \cites{BuzRodTei2022,PesSot2012}. Moreover, there is a great effort of Art\'es and coauthors for the classification of the quadratic vector fields with low codimension. For the classification of quadratic vector fields of codimension zero and one we refer to \cite{ArtKooLli} and \cite{ArtRezLli}. For the ongoing classification of those with codimension two, we refer to \cites{ArtOliRez,ArtMotRez,Art2024}. For the history of Peixoto and the develop of the theory of structural stability, we refer to \cite{SotGarMel2020,Sot2020}. In this paper, we work on the structural instability of non-hyperbolic limit cycles in the set of planar polynomial vector fields.

\section{Statement and discussion of the main results}

As stated in this previous section, the notion of \emph{robustness} of a vector field was first provided by Andronov and Pontrjagin \cite{AndPon1937}. Let $B=\{x\in\mathbb{R}^2\colon ||x||\leqslant1\}$, where $||\cdot||$ denotes the standard norm of $\mathbb{R}^2$, and let $\mathfrak{X}$ be the set of $C^1$-vector fields defined over $B$ and without contact with $\partial B^2$, where $\partial B^2$ denotes the topological boundary of $B^2$. Let also $\mathfrak{X}$ be endowed with the $C^1$-topology (i.e. to vector fields are close if their components and its first order partial derivatives are close).

\begin{definition}[Robustness in the sense of Andronov and Pontrjagin]
	Let $X\in\mathfrak{X}$. We say that $X$ is \emph{robust} if for every $\varepsilon>0$ there is a neighborhood $N\subset\mathfrak{X}$ of $X$ and a family of homeomorphisms $h\colon N\to\text{Hom}(B^2,B^2)$, such that for every $Y\in N$ the homeomorphism $h_Y\colon B^2\to B^2$ sends orbits of $X$ to orbits of $Y$, either preserving of reversing the direction of all orbits, and such that the following statements hold.
	\begin{enumerate}[label=(\alph*)]
		\item $h_X=\text{Id}_{B^2}$;
		\item For every $Y\in N$ and $x\in B^2$, $||h_Y(x)-x||<\varepsilon$.
	\end{enumerate}
\end{definition}

One of the first and main contributions of Peixoto \cite{Pei1959} was to prove the above statements $(a)$ and $(b)$ are redundant (even in the case of analytical vector fields). That is, the homeomorphism $h_Y$ does not need to be in a $\varepsilon$-neighborhood of the identity map. Hence, Peixoto provided a more general, and yet equivalent, notion of robustness. He called such notion \emph{structural stability}.

\begin{definition}[Structural stability in the sense of Peixoto]
	Let $X\in\mathfrak{X}$. We say that $X$ is \emph{structural stable} if there is a neighborhood $N\subset\mathfrak{X}$ of $X$ such that for each $Y\in N$ there is an homeomorphism $h_Y\colon B^2\to B^2$ sending orbits of $X$ to orbits of $Y$, either preserving of reversing the direction of all orbits.
\end{definition}

After that, Peixoto also extended his work by replacing $B^2$ for a two-dimensional manifold $M\subset\mathbb{R}^2$ with boundary and conners \cite{PeiPei1959} and then extended it again allowing $M$ to be any two-dimensional closed manifold \cite{Pei1962}. Similar results also hold for the case in which $M\subset\mathbb{R}^2$ is an open surface (e.g. $M=\mathbb{R}^2$). See \cite{Kot1982}. However, when dealing with structural stability in \emph{polynomial vector fields} of some maximum degree $n$, things become more difficult to tackle. This is the case because to proof that some object \emph{is not} structurally stable (e.g. a non-hyperbolic limit cycle) one need to \emph{construct} some suitable perturbation that ``breaks'' this object (e.g. slip the non-hyperbolic limit cycle in two or more limit cycles). Therefore, when working with a more restrict space of vector fields, one may obtain a more broader set of structurally stable objects. Concerns about this go back to Aldronov et al \cite[$\mathsection6.3$]{And1971}. Still, there are great approaches of the structural stability of polynomial vector fields. For this, we refer the works of Sotomayor \cite{Soto1985} and Shafer \cite{Sha1987}. On one hand, Sotomayor \cite{Soto1985} defined the structural stability of a planar polynomial vector field $X$ as the structural stability of its Poincar\'e compactification $p(X)$ (see \cite[Chapter $5$]{DumLliArt2006}) and endowed the space of vector fields with the coefficients topology. On the other hand, Shafer \cite{Sha1987} approachs $X$ as a vector field defined on the open surface $M=\mathbb{R}^2$ and endow the space of vector fields with either the Whitney's $C^r$-topology, $r\geqslant 1$, or the coefficients topology. However, in the case of coefficients topology, there is an open question that kept both Sotomayor and Shafer from obtaining necessary \emph{and} sufficient conditions for structural stability.

\begin{question}\label{Q1}
	Let $\mathcal{X}_n$ be the set of the planar polynomial vector fields of degree at most $n$, endowed with the coefficients topology. If $X\in\mathcal{X}_n$ has a non-hyperbolic limit cycle of odd degree, then is $X$ structurally unstable in $\mathcal{X}_n$?
\end{question}

Here we recall that the degree of a limit cycle (also known as \emph{multiplicity}) is the order of the first non-zero derivative of its displacement map. Question~\ref{Q1} was explicitly raised by both Sotomayor \cite[Problem $1.1$]{Soto1985} and Shafer \cite[Question $3.4$]{Sha1987}. In the case of non-hyperbolic limit cycles of even degree, the structural instability follows from the theory of rotated vector fields. More precisely, if $X=(P,Q)$ has an non-hyperbolic limit cycle $\gamma$ of even degree, then consider $Y_\lambda=X+\lambda X^\perp$, where $X^\perp=(-Q,P)$. It follows from \cite[Theorem $2$, p. $387$]{Perko2001} that for $|\lambda|>0$ small enough, $\gamma$ had either vanished or slip in two or more limit cycles. In the case of smooth and analytical vector fields it is well known that non-hyperbolic limit cycles (in particular, those with odd degree) are structurally unstable. Briefly, the proof work as follows. Let $X=(P,Q)$ be a planar vector field of class $C^r$, $r\geqslant 1$, and let $\gamma(t)$ be a limit cycle of $X$, with period $T>0$. Andronov et al \cite[p. 124]{And1971}, proved that there is a neighborhood $G\subset\mathbb{R}^2$ of $\gamma$ and a function $F\colon G\to\mathbb{R}$ of class $C^{r+1}$, such that 
	\[F(\gamma(t))=0, \quad \frac{\partial F}{\partial x}(\gamma(t))^2+\frac{\partial F}{\partial y}(\gamma(t))^2>0,\]
for all $t\in[0,T]$. Peixoto \cite[Lemma~$6$]{Pei1959}, proved that if $X$ is analytical, then $F$ is also analytical. With such function, we consider the vector field $Y_\lambda=(R_\lambda,S_\lambda)$ given by
	\[R_\lambda(x,y)=P(x,y)+\lambda F(x,y)\frac{\partial F}{\partial x}(x,y), \quad S_\lambda(x,y)=Q(x,y)+\lambda F(x,y)\frac{\partial F}{\partial y}(x,y),\]
with $\lambda\in\mathbb{R}$. By using the Poincar\'e-Bendixson Theorem, one can prove that for $|\lambda|>0$ small enough, $\gamma$ had split in at least two limit cycles. For more details, see \cite[$\mathsection15.2$]{And1971} and \cite[Lemma~$6$]{Pei1959}. Observe that if $F$ is polynomial, then $\gamma$ is an algebraic limit cycle. Since not every limit cycle is algebraic (see \cite{GasGiaTor2007}), it follows that even if $X$ is polynomial, $F$ may not be. Hence, for the polynomial case the usual tools breaks down. We now state our main results concretely. Given $r\geqslant1$ \emph{finite}, let $C^r(\mathbb{R}^2,\mathbb{R}^2)$ be the set of the functions $f\colon\mathbb{R}^2\to\mathbb{R}^2$ of class $C^r$. Given $f\in C^r(\mathbb{R}^2,\mathbb{R}^2)$, a compact set $K\subset\mathbb{R}^2$, an open set $U\subset\mathbb{R}^2$ and $\varepsilon>0$, let $V(f,K,U,\varepsilon)\subset C^r(\mathbb{R}^2,\mathbb{R}^2)$ be the set of functions $C^r$-functions $g\colon\mathbb{R}^2\to\mathbb{R}^2$ such that $g(K)\subset U$ and 
	\[\max_{\substack{(x,y)\in K \\ |k|\leqslant r}}\left|\left|\frac{\partial^{|k|}f}{\partial x^{k_1}\partial y^{k_2}}(x,y)-\frac{\partial^{|k|}g}{\partial x^{k_1}\partial y^{k_2}}(x,y)\right|\right|<\varepsilon,\]
where $k=(k_1,k_2)\in\mathbb{Z}^2_{\geqslant0}$ and $|k|=k_1+k_2$. The Whitney's weak $C^r$-topology \cite{Hirsch} is the topology on $C^r(\mathbb{R}^2,\mathbb{R}^2)$ having all such $V(f,K,U,\varepsilon)$ as a sub-base. In other words, it is the smaller topology that contains all such $V(f,K,U,\varepsilon)$. Let now $\mathfrak{X}$ be the set of all the planar polynomial vector fields \emph{of any degree}. Since $\mathfrak{X}\subset C^r(\mathbb{R}^2,\mathbb{R}^2)$, we can endow $\mathfrak{X}$ with the \emph{subspace topology} $\tau_r$, inherited from Whitney's weak $C^r$-topology. Let $\mathfrak{X}^r=(\mathfrak{X},\tau_r)$. Observe that two vector fields $X$, $Y\in\mathfrak{X}^r$ are close if there is a ``big'' compact $K\subset\mathbb{R}^2$ and a small $\varepsilon>0$ such that 
	\[\max_{\substack{(x,y)\in K \\ |k|\leqslant r}}\left|\left|\frac{\partial^{|k|}X}{\partial x^{k_1}\partial y^{k_2}}(x,y)-\frac{\partial^{|k|}Y}{\partial x^{k_1}\partial y^{k_2}}(x,y)\right|\right|<\varepsilon.\]
Our first main result is the following.

\begin{theorem}\label{T1}
	Let $X$ be a planar polynomial vector field with a non-hyperbolic limit cycle $\gamma$ of odd degree. Then $X$ is structurally unstable in $\mathfrak{X}^r$, for any $r\geqslant1$ finite. 
\end{theorem}

In simple words, our main result is that non-hyperbolic limit cycles of odd degree are structurally unstable in relation to \emph{polynomials} perturbation on the Whitney's topology. We observe that the author in \cite{Sha1987} already worked with the Whitney's topology, however his perturbations are not polynomial. We observe that in this paper we do not study the codimension of the bifurcation of non-hyperbolic odd limit cycles in $\mathfrak{X}^r$. Rather, we only prove that \emph{there is} a polynomial perturbation. Which is, to the best of our knowledge, a new fact. 

We now state our second main result. Given $n\in\mathbb{N}$, let $\mathcal{X}_n$ be the space of planar polynomial vector fields of degree at most $n$, endowed with the coefficients topology. In this paper we also provide a different proof of the following already-known fact.

\begin{theorem}\label{T2}
	Let $X$ be a planar polynomial vector field with a non-hyperbolic limit cycle of even degree. Then $X$ is structurally unstable in $\mathcal{X}_n$.
\end{theorem}

As commented before, Theorem~\ref{T2} follows as a corollary of the theory of rotated vector fields (see \cite[Section $4.6$]{Perko2001}). However, we have a different proof that may be extended to a general proof that also extends to the non-hyperbolic limit cycles of odd degree and thus providing a definitive answer to Question~\ref{Q1}. More precisely, let $X\in\mathcal{X}_n$ have a non-hyperbolic limit cycle $\gamma$, of whether even or odd degree. In our proof of Theorem~\ref{T2} we were able to prove that there is a neighborhood $N\subset\mathcal{X}_n$ of $X$, in relation to the coefficients topology, and a non-constant analytical map $\Phi\colon N\to\mathbb{R}$ satisfying the following statements.
\begin{enumerate}[label=(\alph*)]
	\item $0$ is not a regular value of $\Phi$;
	\item $X\in \Phi^{-1}(0)$;
	\item If $Y\in N$ has a non-hyperbolic limit cycle near $\gamma$, then $Y\in \Phi^{-1}(0)$;
	\item $\Phi^{-1}(0)$ has zero Lebesgue measure on $N$.
\end{enumerate}
With this map, the following statements hold (see Remark~\ref{Remark1}).
\begin{enumerate}[label=(\roman*)]
	\item If there is $Y\in N$ such that $\Phi(Y)\neq0$, then non-hyperbolic limit cycles of \emph{even} degree are not structurally stable;
	\item If there is $Y\in N$ such that $\Phi(Y)<0$, then non-hyperbolic limit cycles of \emph{odd} degree are not structurally stable.
\end{enumerate}
It is clear that statement $(i)$ follows directly from statement $(d)$ and thus we have a proof of Theorem~\ref{T2}. However, we were not able to prove that there is $Y\in N$ such that $\Phi(Y)<0$. Therefore, statement $(ii)$ remains an open sufficient condition for a positive answer for Question~\ref{Q1}. Since as far as we known this sufficient condition is new, we find it useful to state it in this paper.

\begin{question}\label{Q2}
	Let $\Phi\colon N\to\mathbb{R}$ be given as in this proof of Theorem~\ref{T2}. There is $Y\in N$ such that $\Phi(Y)<0$?
\end{question}

The paper is organized as follows. In Section~\ref{Sec3} we have some preliminaries results. Theorems~\ref{T1} and \ref{T2} are proved in Sections~\ref{Sec4} and \ref{Sec5}, respectively. 

\section{Preliminaries and technical results}\label{Sec3}

\subsection{Whitney's stratification}\label{Sub3.1}

Let $Z\subset\mathbb{R}^n$ be a closed set. An analytical \emph{stratification} of $Z$ is a filtration of $Z$ by closed sets
\[Z=Z_d\supset Z_{d-1}\supset\dots\supset Z_1\supset Z_0,\]
such that $Z_i\backslash Z_{i-1}$ is either empty or an analytical manifold of dimension $i$. Each connected component of $Z_i\backslash Z_{i-1}$ is called a \emph{stratum} of dimension $i$. Thus, $Z$ is the disjoint union of the strata. An analytical \emph{Whitney stratification} is, among other things, a locally finite analytical stratification. That is, given $p\in Z$ there is a neighborhood $U\subset\mathbb{R}^n$ of $p$ such that at most a finite number of strata intersects $U$. A set $Z\subset\mathbb{R}^n$ is \emph{analytic} if there are a finite number of analytical functions $f_1,\dots,f_k\colon\mathbb{R}^n\to\mathbb{R}$, such that
\[Z=\{x\in\mathbb{R}^n\colon f_1(x)=\dots=f_k(x)=0\}.\]

\begin{theorem}[Theorem~$1.2.10$ of \cite{Trot}]\label{T5}
	Every analytic subset of $\mathbb{R}^n$ admits an analytical Whitney stratification.
\end{theorem}

Let $f\colon\mathbb{R}^n\to\mathbb{R}$ be an analytical non-constant function. If $0\in\mathbb{R}$ is a regular value of $f$, then it follows from the Implicit Function Theorem that $f^{-1}(0)\subset\mathbb{R}^n$ is a analytical manifold of codimension $1$. Hence, Theorem~\ref{T5} is stating that if $0$ is not a regular value of $f$, then $f^{-1}(0)$ is yet endowed with some regularity. More precisely, in this case it follows from Theorem~\ref{T5} that
\[f^{-1}(0)=B_1\cup B_2\cup\dots\cup B_n,\]
where the union is disjoint and $B_i$ is an analytical manifold of codimension $i$. Moreover, if we are interested in a particular point $p\in f^{-1}(0)$, then it follows from the locally finite property that we can restrict the domain of $f$ to a neighborhood of $p$ and thus assume that each $B_i$ has at most a finite number os connected components. In particular, we conclude that $f^{-1}(0)$ has zero Lebesgue measure on that neighborhood. For more details in stratification theory, we refer to \cite{Trot}.

\subsection{Bernstein Polynomials of two variables}\label{Sub3.2}

Let $F\colon [0,1]^2\to\mathbb{R}$ be a map of class $C^r$, $r\geqslant0$. The \emph{Bernstein polynomial} associated to $F$ is given by
\begin{equation}\label{18}
	B_{m,n}^F(x,y)=\sum_{r=0}^{m}\sum_{s=0}^{n}F\left(\frac{r}{m},\frac{s}{n}\right)\binom{m}{r}\binom{n}{s}x^ry^s(1-x)^{m-r}(1-y)^{n-s},
\end{equation}
where $\binom{n}{k}=\frac{n!}{k!(n-k)!}$ is the usual \emph{binomial coefficient}. An important property of the Bernstein polynomials is that $B_{n,m}\to F$ uniformly in the $C^r$ topology. More precisely, we have the following theorem. 

\begin{theorem}[Kingsley, \cite{King1949}]
	If $F\colon [0,1]^2\to\mathbb{R}$ is of class $C^r$, $r\geqslant0$ finite, then
		\[\lim\limits_{(n,m)\to\infty}\frac{\partial^{|k|}B_{m,n}^F}{\partial x^{k_1}\partial y^{k_2}}(x,y)=\frac{\partial^{|k|}F}{\partial x^{k_1}\partial y^{k_2}}(x,y),\]
	uniformly in $(x,y)$, where $k=(k_1,k_2)\in\mathbb{Z}_{\geqslant0}^2$, $|k|=k_1+k_2$ and $|k|\leqslant r$. 
\end{theorem}

In particular, we have the following corollary.

\begin{corollary}\label{P3}
	Let $F\colon\mathbb{R}^2\to\mathbb{R}$ be a function of class $C^r$, $r\geqslant0$ finite, and $B\subset\mathbb{R}^2$ a closed ball centered at the origin. Then for every $n\in\mathbb{N}$, there is a polynomial $R_n\colon\mathbb{R}^2\to\mathbb{R}$ such that
		\[\left|\frac{\partial^{|k|}F}{\partial x^i\partial y^j}(x,y)-\frac{\partial^{|k|}R_n}{\partial x^i\partial y^j}(x,y)\right|<\frac{1}{n},\]
	for all $(x,y)\in B$, $k=(i,j)\in\mathbb{Z}_{\geqslant0}^2$, $|k|=i+j$ and $|k|\leqslant r$. 
\end{corollary}

\subsection{Critical region of a odd limit cycle}\label{Sec3.3}

Given a compact set $B\subset\mathbb{R}^2$, let $\text{Int}(B)$ denote its topological interior. Given a continuous simple closed curve $S\subset\mathbb{R}^2$, we say that $S$ is \emph{piecewise smooth} if it is of class $C^1$ except, perhaps, in at most a finite number of points.

\begin{proposition}\label{P1}
	Let $X\in\mathfrak{X}^r$, $r\geqslant1$ finite, have a stable limit cycle $\gamma$ of odd degree and let $B\subset\mathbb{R}^2$ be a compact such that $\gamma\subset\text{Int}(B)$. Then there are two continuous and piecewise smooth simple closed curves $S_1$, $S_2\subset\text{Int}(B)$, with $S_1$ (resp. $S_2$) in the interior (resp. exterior) region of $B\backslash\gamma$, such that if $\Omega\subset\text{Int}(B)$ is the open region bounded by $S_1$ and $S_2$, then following statements hold.
	\begin{enumerate}[label=(\alph*)]
		\item $X$ is transversal to $S_1$ and $S_2$ and points towards $\Omega$;
		\item $\gamma\subset\Omega$;
		\item There is no singularity of $X$ in $\Omega$.
	\end{enumerate}
	Moreover, there is a neighborhood $N\subset\mathfrak{X}^r$ of $X$ such that every $Y\in N$ is transversal do $S_1$ and $S_2$ and points in the same direction as $X$.
\end{proposition}

\begin{proof} Let $\ell$ be small transversal section of $\gamma$, endowed with a metric $\xi$ such that $\xi=0$ at $\gamma$, $\xi<0$ in the interior component of $B\backslash\gamma$ and $\xi>0$ in the exterior component of $B\backslash\gamma$. Let $\pi\colon\ell\to\ell$ be the first return map associated to $\gamma$. Since $\gamma$ is stable, it follows that if $\xi<0$ (resp. $\xi>0$), then $\xi<\pi(\xi)<0$ (resp. $0<\pi(\xi)<\xi)$. See Figure~\ref{Fig1}.	
\begin{figure}[ht]
	\begin{center}
		\begin{overpic}[height=5cm]{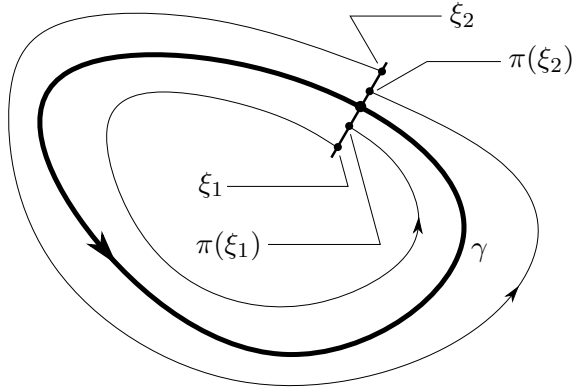} 
			\put(36,37){$\xi_1$}
			\put(35.5,25.5){$\pi(\xi_1)$}
			\put(83,69){$\xi_2$}
			\put(94,60.5){$\pi(\xi_2)$}
			\put(87,25){$\gamma$}
		\end{overpic}
	\end{center}
	\caption{Illustration of the first return map $\pi$, with $\xi_1<0$ and $\xi_2>0$.}\label{Fig1}
\end{figure}
Let $\xi_1<0$. It follows from the continuous dependence of initial conditions (see \cite[Theorem $8$, p. $25$]{And1971}) that there is a neighborhood $N_0\subset\mathfrak{X}^r$ of $X$ such that for every $Y\in N_0$, the orbit $\gamma_Y$ of $Y$ through $\xi_1$ will intersect $\ell$ in a point $\xi_Y$ such that $\xi_1<\xi_Y<0$ (observe that the existence of $\xi_Y$ does not depend on the bifurcations that may occur on $\gamma$). Let $X=(P,Q)$ and define $X^\perp=(-Q,P)$. Let also $X_\lambda=X+\lambda X^\perp$, $\lambda\in\mathbb{R}$, and observe that $X_\lambda\in N_0$ for $|\lambda|$ small enough. Observe that $X_\lambda$ is a small rotation of $X$. More precisely, if $\lambda>0$ (resp. $\lambda<0$) then $X_\lambda$ is a rotation of $X$ in the counterclockwise (resp. clockwise) direction. Let $|\lambda_0|>0$ be small enough such that $X_{\lambda_0}\in N_0$ and denote $Z=X_{\lambda_0}$. Let $\gamma_Z$ be the orbit of $Z$ from $\xi_1$ to $\xi_Z$ and let also $J_1\subset\ell$ be the segment between $\xi_1$ and $\xi_Z$. Let $S_1=\gamma_Z\cup J_1$ and observe that $S_1$ is a continuous simple closed curve. Moreover, observe that $S_1$ is $C^1$ except at $\xi_1$ and $\xi_Z$. Hence, $S_1$ is also piecewise smooth. Since $Z$ is a small rotation of $X$, it follows that $X$ is transversal to $\gamma_Z$. Moreover, it follows from the first return map $\pi\colon\ell\to\ell$ that $X$ is also transversal to $J_1$. Hence, $X$ is transversal to $S_1$. The fact that $X(s)$ points towards $\gamma$ for every $s\in S_1$ follows from an appropriated choice of the sign of $\lambda_0$ (depending on the orientation of $\gamma$). Similarly, one can take $\xi_2>0$ and construct the curve $S_2$ such that $X$ is transversal to $S_2$ and $X(s)$ points towards $\gamma$ for every $s\in S_2$. Hence, we have statement $(a)$. Statement $(b)$ follows by taking a small enough neighborhood $N\subset\mathfrak{X}^r$ of $X$. Statement $(c)$ follows directly from the definition of $\Omega$. The last statement follows by continuity and from the fact that $S_1$ and $S_2$ are compact. \end{proof}

\section{Proof of Theorem~\ref{T1}}\label{Sec4}

Let $\gamma(t)$ be the parametrization of $\gamma$, given by the flow of $X$, and let $T>0$ be its period. Reversing the time variable if necessary, we can assume that $\gamma$ is stable. It follows from \cite[p. 124]{And1971} that there is a neighborhood $G\subset\mathbb{R}^2$ of $\gamma$ and a function $F\colon G\to\mathbb{R}$ of class $C^\infty$ such that
\begin{equation}\label{4}
	F(\gamma(t))=0, \quad \frac{\partial F}{\partial x}(\gamma(t))^2+\frac{\partial F}{\partial y}(\gamma(t))^2>0,
\end{equation}
for all $t\in[0,T]$. Using bump-functions we can assume that $F$ is defined on the entire plane. Let $B\subset\mathbb{R}$ be a closed ball such that $\gamma\subset\text{Int}(B)$. Given $r\geqslant1$ finite, it follows from Corollary~\ref{P3} that for each $n\in\mathbb{N}$ there is a polynomial $R_n\colon\mathbb{R}^2\to\mathbb{R}$ such that
	\[\left|\frac{\partial^{|k|}F}{\partial x^i\partial y^j}(x,y)-\frac{\partial^{|k|}R_n}{\partial x^i\partial y^j}(x,y)\right|<\frac{1}{n},\]
for all $(x,y)\in B$, $k=(i,j)\in\mathbb{Z}_{\geqslant0}^2$, $|k|=i+j$ and $|k|\leqslant r+1$. Let $X_{\lambda,n}=(P_{\lambda,n},Q_{\lambda,n})$ be the two-parameter family of planar polynomial vector field given by,
	\[P_{\lambda,n}(x,y)=P(x,y)+\lambda R_n(x,y)\frac{\partial R_n}{\partial x}(x,y), \quad Q_{\lambda,n}(x,y)=Q(x,y)+\lambda R_n(x,y)\frac{\partial R_n}{\partial y}(x,y),\]
for $\lambda>0$ small. Since $R_n\to F$ in the $C^{r+1}$-topology, restricted at $B$, as $n\to\infty$, we also define $X_{\lambda,\infty}=(P_{\lambda,\infty},Q_{\lambda,\infty})$ as
\begin{equation}\label{21}
	P_{\lambda,\infty}=P(x,y)+\lambda F(x,y)\frac{\partial F}{\partial x}(x,y), \quad Q_{\lambda,\infty}=Q(x,y)+\lambda F(x,y)\frac{\partial F}{\partial y}(x,y).
\end{equation}
Since $R_n\to F$ in the $C^{r+1}$-topology, restricted at $B$, it is clear that $X_{\lambda,n}\to X$ in $\mathfrak{X}^r$. More precisely, given any neighborhood $N\subset\mathfrak{X}^r$ of $X$, there is $\lambda_0>0$ and $n_0\in\mathbb{N}$ such that $X_{n,\lambda}\in N$, for any $0<\lambda<\lambda_0$ and $n\geqslant n_0$. Therefore, since $\gamma$ is not semi-stable, it follows that the perturbation $\gamma_{\lambda,n}$ of $\gamma$ is well defined and still a periodic orbit, for $\lambda>0$ small enough and $n\in\mathbb{N}$ big enough. Observe that,
\begin{equation}\label{6}
	\int_{\gamma_{\lambda,n}}\frac{\partial P_{\lambda,n}}{\partial x}+\frac{\partial Q_{\lambda,n}}{\partial y}=\int_{\gamma_{\lambda,n}}\frac{\partial P}{\partial x}+\frac{\partial Q}{\partial y}+\lambda\int_{\gamma_{\lambda,n}}\left(\frac{\partial R_n}{\partial x}\right)^2+\left(\frac{\partial R_n}{\partial y}\right)^2+\lambda\int_{\gamma_{\lambda,n}} R_n\frac{\partial^2R_n}{\partial x^2}+R_n\frac{\partial^2 R_n}{\partial y^2}.
\end{equation}
In what follows we fix $\lambda>0$. We claim that,
	\[\lim\limits_{n\to\infty}\int_{\gamma_{\lambda,n}}\frac{\partial P}{\partial x}+\frac{\partial Q}{\partial y}=\int_{\gamma}\frac{\partial P}{\partial x}+\frac{\partial Q}{\partial y}=0.\]
Indeed, let $T_{\lambda,n}$ be the period of $\gamma_{\lambda,n}$. Since $F(\gamma)=0$, it follows that $\gamma$ is also a periodic orbit of $X_{\lambda,\infty}$ and thus $\gamma_{\lambda,n}\to\gamma$ uniformly and $T_{\lambda,n}\to T$, as $n\to\infty$. Hence,
	\[\begin{array}{rl}
		\displaystyle \lim\limits_{n\to\infty}\int_{\gamma_{\lambda,n}}\frac{\partial P}{\partial x}+\frac{\partial Q}{\partial y} &\displaystyle= \lim\limits_{n\to\infty}\int_{0}^{T_{\lambda,n}}\left(\frac{\partial P}{\partial x}+\frac{\partial Q}{\partial x}\right)(\gamma_{\lambda,n}(t))\;dt \vspace{0.2cm} \\
		&\displaystyle= \lim\limits_{n\to\infty}\int_{0}^{T}\left(\frac{\partial P}{\partial x}+\frac{\partial Q}{\partial x}\right)(\gamma_{\lambda,n}(t))\;dt+\lim\limits_{n\to\infty}\int_{T}^{T_{\lambda,n}}\left(\frac{\partial P}{\partial x}+\frac{\partial Q}{\partial x}\right)(\gamma_{\lambda,n}(t))\;dt\vspace{0.2cm} \\
		&\displaystyle=\int_{0}^{T}\left(\frac{\partial P}{\partial x}+\frac{\partial Q}{\partial x}\right)(\gamma(t))\;dt=\int_{\gamma}\frac{\partial P}{\partial x}+\frac{\partial Q}{\partial x}=0,
	\end{array}\]
with the last equality following from the fact that $\gamma$ is a non-hyperbolic limit cycle of $X$ and the equality before that following from the uniform convergence $\gamma_{\lambda,n}\to\gamma$. Similarly, since $R_n\to F$, $\frac{\partial^2R_n}{\partial x^2}\to\frac{\partial^2F}{\partial x^2}$ and $\frac{\partial^2R_n}{\partial y^2}\to\frac{\partial^2F}{\partial y^2}$ uniformly in $B$ as $n\to\infty$, it follows that
	\[\lim\limits_{n\to\infty}\int_{\gamma_{\lambda,n}} R_n\frac{\partial^2R_n}{\partial x^2}+R_n\frac{\partial^2 R_n}{\partial y^2}=\int_{\gamma}F\frac{\partial^2 F}{\partial x^2}+F\frac{\partial^2 F}{\partial y^2}=0,\]
with the last equality following from $F(\gamma)=0$. Moreover, it also follows from \eqref{4} that,
	\[\lim\limits_{n\to\infty}\int_{\gamma_{\lambda,n}}\left(\frac{\partial R_n}{\partial x}\right)^2+\left(\frac{\partial R_n}{\partial y}\right)^2=\int_{\gamma}\left(\frac{\partial F}{\partial x}\right)^2+\left(\frac{\partial F}{\partial y}\right)^2>0.\]
Therefore, given $\lambda>0$ small, it follows from \eqref{6} that
\begin{equation}\label{7}
	\int_{\gamma_{\lambda,n}}\frac{\partial P_{\lambda,n}}{\partial x}+\frac{\partial Q_{\lambda,n}}{\partial y}>0,
\end{equation}
for $n\in\mathbb{N}$ big enough. Hence, we conclude that for every $\lambda>0$ there is $n_\lambda\in\mathbb{N}$ such that if $n\geqslant n_\lambda$, then $\gamma_{\lambda,n}$ is a unstable hyperbolic limit cycle of $X_{\lambda,n}$. We claim that $X_{\lambda,n}$ has at least two others limit cycles near $\gamma_{\lambda,n}$. Indeed, let $N\subset\mathfrak{X}^r$ and $\Omega\subset\text{Int}(B)$ be given by Proposition~\ref{P1}. Observe $\Omega$ is invariant by the flow of every $Y\in N$. Let $\lambda>0$ be small enough and $n\geqslant n_\lambda$ be big enough such that $X_{\lambda,n}\in N$. Since $\gamma_{\lambda,n}\subset\Omega$ is a unstable limit cycle, it follows from the Poincare-Bendixson Theorem that $X_{\lambda,n}$ have at least two others limit cycles in $\Omega$, one in each connected component of $\Omega\backslash\gamma_{\lambda,n}$. The proof now follows from the fact that $N\subset\mathfrak{X}^r$ can be taken arbitrarily small. {\hfill$\square$}

\section{Proof of Theorem~\ref{T2}}\label{Sec5}

\begin{proof} Let $\gamma$ be a non-hyperbolic limit cycle for $X\in\mathcal{X}_n$. Let also $N\subset\mathfrak{X}_n$ be a small enough neighborhood of $X$ such that the displacement map $D\colon N\times\ell\to\mathbb{R}$ is well defined, where $\ell$ is a transversal section of $\gamma$. Let $d$ be the degree of $\gamma$. That is, $d\geqslant2$ is the first integer such that,
	\[\frac{\partial^d D}{\partial x^d}(X,0)\neq0.\]
It follows from the Weierstrass Preparation Theorem (see \cite[Chapter $4$]{GolGui1973}) that $D(Y,x)=U(Y,x)P(Y,x)$, where $U$ is a strictly positive analytical function and
	\[P(Y,x)=x^d+a_{d-1}(Y)x^{d-1}+\dots+a_1(Y)x+a_0(Y),\]
where $a_i$ is analytical and $a_i(X)=0$, $i\in\{0,\dots,d-1\}$. Given $\varepsilon>0$, let $X_\lambda$ be the $1$-parameter family given by
\begin{equation}\label{1}
	P_\lambda(x,y)=P(x,y)-\lambda\varepsilon Q(x,y) \quad Q_\lambda(x,y)=Q(x,y)+\lambda\varepsilon P(x,y),
\end{equation}
with $\lambda\in(-1,1)$. Observe that if $\varepsilon>0$ is small enough, then $X_\lambda\in N$, for all $\lambda\in(-1,1)$. Observe also that $X_0=X$. Let $\gamma(t)$ be a parametrization of $\gamma$ such that $\gamma(0)=p$, where $\{p\}=\ell\cap\gamma$. Let $T>0$ be the period of $\gamma$. It follows from \cite[Lemma $2$]{Perko1992} that,
\begin{equation}\label{5}
	\frac{\partial D}{\partial\lambda}(X_0,0)=C\int_{0}^{T}\left(e^{-\int_{0}^{t}Div(\gamma(s))\;ds}\right)X_0(\gamma(t))\land\dfrac{\partial X_0}{\partial \lambda}(\gamma(t))\;dt,
\end{equation}
with $C\in\mathbb{R}\backslash\{0\}$ and $(x_1,x_2)\land(y_1,y_2)=x_1y_2-x_2y_1$. Observe that, 
	\[a_0(Y)=\frac{D(Y,0)}{U(Y,0)},\]
for any $Y\in N$. To simplify the notation, let $V(Y,x)=U(Y,x)^{-1}$. Hence,
	\[\frac{\partial a_0}{\partial\lambda}(X_0)=\frac{\partial V}{\partial\lambda}(X_0,0)\underbrace{D(X_0,0)}_{0}+V(X_0,0)\frac{\partial D}{\partial\lambda}(X_0,0).\]
Thus, it follows from \eqref{5} that,
\begin{equation}\label{3}
	\frac{\partial a_0}{\partial\lambda}(X_0)=V(X_0,0)C\int_{0}^{T}\left(e^{-\int_{0}^{t}Div(\gamma(s))\;ds}\right)X_0(\gamma(t))\land\dfrac{\partial X_0}{\partial \lambda}(\gamma(t))\;dt.
\end{equation}
It follows from \eqref{1} that,
	\[X_0(x,y)\land\dfrac{\partial X_0}{\partial \lambda}(x,y)=\varepsilon\bigl(P(x,y)^2+Q(x,y)^2\bigr).\]
Hence
	\[X_0(\gamma(t))\land\dfrac{\partial X_0}{\partial \lambda}(\gamma(t))>0,\]
for all $t\in[0,T]$. Thus, it follows from \eqref{3} that,
	\[\frac{\partial a_0}{\partial\lambda}(X)\neq0.\]
Therefore, the function $a_0\colon N\to\mathbb{R}$ has a non-zero directional derivative and thus it is not constant. In particular, the polynomial
	\[P(Y,x)=x^d+a_{d-1}(Y)x^{d-1}+\dots+a_1(Y)x+a_0(Y),\]
is not constant in $Y$. From now on, we denote $P(Y)=P(Y,\cdot)$. Let $\Delta\colon\mathbb{R}^{d+1}\to\mathbb{R}$ denote the discriminant of polynomials of degree $d$ (see \cite[Chapter $12$]{GelKapZel1994}). Observe that if $\gamma(Y)$ is a non-hyperbolic limit cycle of $Y$, then $\Delta(P(Y))=0$. Since $\Delta\circ P$ is non-constant, it follows from the stratification theory (recall Section~\ref{Sub3.1}) that the set 
	\[\Omega=\{Y\in N\colon \Delta(P(Y))=0\},\]
is given by the disjoint union of analytical manifolds of codimension at least one, each of them having at most a finite number of connected components. In particular, $\Omega$ has zero Lebesgue measure on $N$ and thus we can take $Y$ arbitrarily close to $X$ such that $\Delta(P(Y))\neq0$. Therefore, if $d$ is even, then $P(Y)$ has either zero or at least two real roots. In the former, $Y$ has no limit cycles near $\gamma$. In the latter, $Y$ has at least two hyperbolic limit cycles near $\gamma$. In either case, we have proved that $\gamma$ is structurally unstable. \end{proof}

\begin{remark}\label{Remark1}
	In the context of the proof of Theorem~\ref{T2}, if $d$ is odd, then it follows from \cite[Section $4$]{NickDye} that the following statements hold.
	\begin{enumerate}[label=(\roman*)]
		\item If $\Delta(P(Y))>0$, then the number of real roots of $P$ is congruent to $1$ modulo $4$;
		\item If $\Delta(P(Y))<0$, then the number of real roots of $P$ is congruent to $3$ modulo $4$.
	\end{enumerate}
	Therefore, if there is $Y\in N$ such that $\Delta(P(Y))<0$, then $Y$ has at least three hyperbolic limit cycles near $\gamma$ and thus we conclude that non-hyperbolic limit cycles of odd degree are also structurally unstable in $\mathcal{X}_n$. Hence, by taking $\Phi=\Delta\circ P$ we obtain Question~\ref{Q2}.
\end{remark}

\section*{Acknowledgments}

We thank to the reviewers their comments and suggestions which help us to improve the presentation of this paper. The author is supported by S\~ao Paulo Research Foundation (FAPESP), grants 2019/10269-3, 2021/01799-9 and 2022/14353-1.

\section*{Declarations}

\noindent\textbf{Competing interests:} On behalf of all authors, the corresponding author states that there is no conflict of interest.

\end{document}